\documentclass[12pt]{amsart}

\usepackage[utf8]{inputenc}
\usepackage[english]{babel}
\usepackage{amssymb}
\usepackage{amsmath}

\usepackage[
backend=biber,
style=alphabetic,
sorting=nyt
]{biblatex}

\usepackage{url}
\usepackage{breakurl} 
\usepackage{hyperref}

\usepackage{xcolor}

\newcommand{\be}{\begin{equation}}
\newcommand{\ee}{\end{equation}}

\newcommand{\N}{\mathbb{N}}

\newtheorem{theorem}{Theorem} [section]
\newtheorem{corollary}[theorem]{Corollary}
\newtheorem{proposition}[theorem]{Proposition}
\newtheorem{remark}[theorem]{Remark}
\newtheorem{lemma}[theorem]{Lemma}

\usepackage{csquotes}

\begin{document}


\thanks{Supported by ANPCyT under grant PICT-2018-03017, and by Universidad de
	Buenos Aires under grant 20020160100002BA. The author is a members of
	CONICET, Argentina.}

\title{An Identity involving the Cyclotomic Polynomials}
\author{Pablo L.  De N\'apoli}
\address{IMAS (UBA-CONICET) and Departamento de Matem\'atica, Facultad de Ciencias Exactas y Naturales, Universidad de Buenos Aires, Ciudad Universitaria, 1428 Buenos Aires, Argentina}
\email{pdenapo@dm.uba.ar}
\maketitle

\begin{abstract}
We present an elementary identity for the cyclotomic polynomials $\Phi_n(x)$ which reflects 
a kind of multiplicative property of $\Phi_n(x)$ as a function of $n$, and we explore its connections with 
the properties of other arithmetical functions. \\

Important Note: In the first version of this article uploaded to the arXiv, it is said that this result  seemed to be new. However, after that, I have learned that this identity in Theorem \ref{theorem-our-identiy} has previously appeared in \cite[corollary 2]{Cheng} (with a different proof).  
\end{abstract}

\section{Introduction and main result}

For each natural number $n \in \N$, let $\Phi_n(X)$ denote the $n$-th cyclotomic polynomial, i.e.: the monic polynomial whose roots are the primitive $n$-th roots of unit. Explicitly 
$$ \Phi_n(X) = \prod_{\stackrel{1 \leq k \leq n}{k \perp n}} \left(X - \zeta_k \right), \quad \zeta_k=  
e^{\frac{2\pi i k}{n}}. $$
Here $k \perp n$ means that $m$ and $n$ are \emph{coprime} or \emph{relatively prime} (a useful notation introduced in \cite[section 4.5]{CM}), i.e.
$$ k \perp n \Leftrightarrow \gcd(k,n)=1 $$

The cyclotomic polynomials are a well-known object in number theory and they also play a key role in field theory, 
see for instance \cite[section 13.6]{Dummit-Foote}. Nice surveys on the subject of cyclotomic polynomials are \cite{thangadurai2000coefficients} and \cite{ge2008elementary}. Also the web page \cite{weisstein2002polynomial} collects some known results and references on them. 

\medskip

In particular, it is known that the cyclotomic polynomials $\Phi_n(X)$ have integral coefficients, and that many important functions in multiplicative number theory are related to them.

\medskip

For instance, it is clear from their definition that their degree is given by $\varphi(n)$, 
Euler's totient function, which counts the number of integers $k$ in the range $1\leq k \leq n$ that are coprime with $n$. Also it is easily seen that the cyclotomic polynomials satisfy the following \emph{Fundamental identity}
\be X^n -1 = \prod_{d|n} \Phi_d(X)  \label{fundamental-identity} \ee
(since every $n$-th root of the unit is a $d$-th primitive root for exactly one $d$ dividing $n$), 
from where we deduce that the cyclotomic polynomials can be computed recursively using the formula
\be \Phi_n(X) = \frac{X^n-1}{\prod_{d|n,d<n} \Phi_d(X) }. \label{recursion} \ee

\medskip

Moreover, if we consider the coefficients $a_k(n)$ of $\Phi_n(X)$, i.e. we write
\be \Phi_n(X) = \sum_{k=0}^{\varphi(n)} a_k(n) X^k \label{def-coefficients}, \ee
we have that
$$ a_1(n)= a_{\varphi(n)-1}(n) = -\mu(n) \; \quad \hbox{for} \; n> 1 $$
where $\mu$ is the Möbius function 
$$ \mu(n) = \left\{
\begin{array}{ll}
1 & \hbox{if} \; n=1 \\
(-1)^k & \hbox{if} \; n=p_1 p_2 \ldots p_k \; \hbox{for distinct primes} \; p_j.\\ 
0 & \hbox{otherwise} \\
\end{array}\right. $$
Indeed, it is well known that $\mu(n)$ gives the sum of the $n$-th primitive roots of the unit so that
$$ a_{\varphi(n)-1}(n) = -\mu(n) $$
and $a_1(n)= a_{\varphi(n)-1}(n)$ by the symmetry of the cyclotomic polynomial \cite[Lemma 2.1]{thangadurai2000coefficients}.

Also from the Fundamental Identity \eqref{fundamental-identity} we can derive the expression
$$ \Phi_n(X) = \prod_{d|n} (X^{n/d}-1)^{\mu(d)} = \prod_{d|n} (X^d-1)^{\mu(n/d)} $$
using the multiplicative version of Möbius inversion formula (see Lemma \ref{lemma-Moebius-inversion} below).

\medskip

Many of the \emph{arithmetical functions}  in multiplicative number theory  $f:\N \to R$ (where R is some 
commutative ring, usually the field $\mathbb{C}$ of complex numbers) are \emph{multiplicative} in the sense that
$$ f(m\cdot n)= f(m)\cdot f(n) \; \hbox{whenever} \; m \perp n $$ 
For instance $\varphi$ and $\mu$ have this property 
(see \cite{HR} or \cite[chapter 2]{Apostol}, and section \ref{section-phi-multiplicative} below). 

\medskip 

Another arithmetical function closely related to the cyclotomic polynomials is the \emph{Ramanujan sum} $c_n(q)$  (introduced in \cite{Ramanujan1918}), defined as the sum of the $q$-powers of the $n$-th primitive roots of the unit 
	\be c_n(q) = \sum_{\stackrel{1 \leq k \leq n}{k \perp n}} \zeta_k^q\;, \quad \zeta_k=  
e^{\frac{2\pi i k}{n}} \label{Ramanujan-sums}. \ee

As we have mentioned before 
\be \mu(n)= c_n(1) \; \hbox{for all} \; n \in \N \label{mu-Ramanujan}. \ee 

The Ramanujan sums are multiplicative as a function of $n$ (see 
\cite[Theorem 67]{HR})
$$ c_{mn}(q) = c_{m}(q) \cdot c_n(q) \;  \hbox{whenever} \; m \perp n $$
and also satisfy the following more complex multiplicative property \cite[Theorem 8.7]{Apostol}
$$ c_{mn}(ab) = c_{m}(a)\cdot c_{n}(b) \; \hbox{whenever} \; a \perp n \; \hbox{and} \; b \perp m. $$

Likewise, other arithmetical functions defined by sums involving the 
roots of the unit, like Gauss quadratic sums and Kloosterman sums, enjoy similar multiplicative properties \cite[section 5.6]{HR}. 

\medskip

A question that naturally arises is whether $\Phi_n(X)$, considered as an arithmetical function of $n \in \N$ into the ring $\mathbb{Z}[X]$ of polynomials with integral coefficients, has some property of this kind. 

\medskip

In this note, we present an elementary identity involving the cyclotomic polynomials, answering this question.

\medskip

In the first version of this article uploaded to the Arxiv, it is said that this result  seemed to be new. However, after that, I have learned that this result has previously appeared in \cite[corollary 2]{Cheng}(with a different proof). 

\begin{theorem}
Let $m$ and $n$ be coprime. Then,
\be \Phi_n(X^m) = \prod_{d|m} \Phi_{d\cdot n}(X) \label{our-identity} \ee
\label{theorem-our-identiy}
\end{theorem}
\begin{proof}
The proof uses (complete) induction on $n$ . We will show that for each $n \in \N$, \eqref{our-identity} holds
for every $m\in \N$ such that $m \perp n$.

\medskip

Indeed for $n=1$, $\Phi_1(X)=X-1$ and \eqref{our-identity} reduces to the Fundamental Identity \eqref{fundamental-identity}.

\medskip

Next, we assume then that  \eqref{our-identity} holds for any $n^\prime<n$ in place of $n$, 
and we will show that it holds for $n$.

From \eqref{recursion} (substituting $X^m$ for $X$), we have that
$$ \Phi_n(X^m) = \frac{(X^{m})^n-1}{\prod_{d_2|n,d_2<n} \Phi_{d_2}(X^m) }, $$ 
and using the inductive hypothesis (with $n^\prime=d_2$), this can be written as 
\be \Phi_n(X^m) = \frac{(X^{m})^n-1}{\prod_{d_2|n,d_2<n} \prod_{d_1|m} \Phi_{d_1\cdot d_2}(X) }. \label{step-one} \ee 
Here we have used the fact that since $d_2|n$, $d_2$ is also coprime with $m$.

\medskip

On the other hand, from the Fundamental Identity \eqref{fundamental-identity},
$$ (X^{m})^n-1 = X^{mn} - 1 = \prod_{d|mn}  \Phi_d(X). $$

Now we observe that the Fundamental Theorem of Arithmetic implies that, since $m$ and $n$ are coprime,
each divisor $d$ of $mn$ can be uniquely decomposed as
$$ d=d_1\cdot d_2 \; \hbox{where} \; d_1|m \; \hbox{and} \; d_2|n. $$
Hence, we can write
\begin{align*}
(X^{m})^n-1 &= \prod_{d_1|m} \prod_{d_2|n} \Phi_{d_1\cdot d_2}(X) \\ 
& = \left[ \prod_{d_1|m} \prod_{d_2|n,d_2<n} \Phi_{d_1\cdot d_2}(X) \right]
\cdot \left[ \prod_{d_1|m}  \Phi_{d_1\cdot n}(X) \right]   
\end{align*}

(splitting the factor with $d_2=n$). Replacing in \eqref{step-one}, it follows that
$$ \Phi_n(X^m) = \prod_{d_1|m}  \Phi_{d_1\cdot n}(X)  $$ 
as claimed. 

By the principle of (complete) mathematical induction it follows that the theorem holds for every $n,m \in \N$.
\end{proof}

\begin{remark}
It is easily seen that \eqref{our-identity} fails if $n$ and $m$ are not coprime. For instance if $m=2$
and $n=4$
$$ \Phi_4(X)=X^2+1 \Rightarrow \Phi_n(X^m) = \Phi_4(X^2)=  X^4+1, $$
whereas
$$ \prod_{d|m} \phi_{d\cdot n}(X)= \Phi_{4}(X) \Phi_{8}(X) = (X^2+1)(X^4+1)= X^{6} + X^{4} + X^{2} + 1 $$
It is my pleasure to acknowledge that the software \emph{Sagemath} \cite{sagemath} was used to find this counterexample and to check many of the identities in this work, and to thank their developers for this wonderful tool. 
\end{remark}

\section{A dual form of the main identity}

In this section, we prove a dual form of our main indentity.

\begin{theorem}
If $n$ and $m$ are coprime,
\be \Phi_{nm}(X) = \prod_{d|m} \Phi_{n}(X^d)^{\mu(m/d)}
= \prod_{c|m} \phi_{n}(X^{m/c})^{\mu(c)} \label{dual-form} \ee
\label{theorem-dual-form}
\end{theorem}

For the proof we need a the Möbius inversion formula that we state as a  lemma (see \cite[theorems 266 and 267]{HR} for a proof).

\begin{lemma}[Möbius inversion formula]
Let $f,g:\N \to R$ be two functions, where $R$ is a commutative ring. 
\begin{enumerate} 
\item[i)] (Additive form)
The relation 
$$ g(m) = \sum_{d|m} f(d)  \; \hbox{for every} \; m \in \N $$
is equivalent to
$$ f(m) = \sum_{d|m} g(d) \mu\left( \frac{m}{d} \right) = \sum_{c|m} g\left(\frac{m}{c}\right) \mu(c) \; \hbox{for every} \; m \in \N $$
\item[ii)] (Multiplicative form)
Assume that $R$ is a field. Then, the relation 
\be g(m) = \prod_{d|m} f(d) \label{hypothesis-inversion}   \; \hbox{for every} \; m \in \N \ee
is equivalent to 
\be f(m) = \prod_{d|m} g(d)^{\mu(m/d)} = \prod_{c|m} g(m/c)^{\mu(c)} \label{thesis-inversion}   \; \hbox{for every} \; m \in \N \ee 
Here we make the convention that $x^0=1$ even if $x=0$.
\end{enumerate}
\label{lemma-Moebius-inversion}
\end{lemma}

In our application of the multiplicative form of Möbius inversion formula, $R=\mathbb{Q}(x)$ is the field of rational functions with rational coefficients. Now we see that using the lemma, Theorem \ref{theorem-dual-form} follows from Theorem \ref{theorem-our-identiy} by fixing $n$ and considering
$$ f(d) = \left\{
\begin{array}{rcl}
\Phi_{dn}(x) & \hbox{if} \; d \perp n \\ 
0 &   \hbox{otherwise}
\end{array}
\right.$$

$$ g(m) =
\left\{
\begin{array}{rcl}
\Phi_n(X^m) & \hbox{if} \; m \perp n \\ 
0 &   \hbox{otherwise}
\end{array}
\right.
$$

The relation  \eqref{hypothesis-inversion} is just \eqref{our-identity} if $n \perp m$. Likewise \eqref{thesis-inversion} reduces to \eqref{dual-form} when $m \perp n$ as $d|m$ implies that $d\perp n$. If not, both sides of \eqref{hypothesis-inversion} vanish as $d=m$ is one of the divisors in the right hand side.

\medskip

Some known properties of the cyclotomic polynomial follow easily from our identity.

\begin{corollary} \cite[Corollary 2.3]{ge2008elementary}
If $p$ is a prime and $k\geq 1$ then,
$$
\Phi_{p^k \cdot n}(X) =
\left\{
\begin{array}{ll}  
 \Phi_n(X^{p^k} )&  \hbox{if} \;  p \;\hbox{divides} \;n \\
 \frac{\Phi_{n}(X^{p^k})}{\Phi_{n}(X^{p^{k-1}})} &  \hbox{if} \;  p \;\hbox{does not divide} \;n. \\
\end{array}
\right. $$

\end{corollary}

\begin{proof} We first consider the case in which $p$ does not divide $n$. We use theorem 
\ref{theorem-dual-form} with $m=p^k$.
$$ \Phi_{p^kn}(X) =  \prod_{c|p^k} \Phi_{n}(X^{m/c})^{\mu(c)}
=  \prod_{j=0}^k \Phi_{n}(X^{m/p^j})^{\mu(p^j)}
=  \frac{\Phi_{n}(X^{p^k})}{\Phi_{n}(X^{p^{k-1}})} $$
since by the definition of the Möbius function 
$$ \mu(p^j)= \left\{
\begin{array}{lll}
1 & \hbox{for} & j=0 \\
-1 & \hbox{for} & j=1 \\
0 & \hbox{for} & j\geq 2
\end{array} 
\right.,$$
this proves the corollary in this case.
\medskip

If $p$ divides $n$, we write $n=p^j \cdot n^\prime$ where $p$ does not divide $n^\prime$. Then, using what we 
have already proved, 
$$ \Phi_{p^k \cdot n}(X) = \Phi_{p^{k+j} \cdot n^\prime}(X)
=   \frac{\Phi_{n^\prime}(X^{p^{k+j}})}{\Phi_{n^\prime}(X^{p^{k+j-1}})}. $$
Likewise
$$  \Phi_{n}(X)= \Phi_{p^k \cdot n^\prime}(X)= \frac{\Phi_{n^\prime}(X^{p^{k}})}{\Phi_{n^\prime}(X^{p^{k-1}})}. $$
Then, substituting $X^{p^j}$ for $X$,
$$ \Phi_{n}(X^{p^j}) =  \frac{\Phi_{n^\prime}(X^{p^j})^{p^{k}})}{\Phi_{n^\prime}((X^{p^j})^{p^{k-1}})}  = 
\frac{\Phi_{n^\prime}((X^{p^{k+j}})}{\Phi_{n^\prime}(X^{p^{k+j-1}})} = \Phi_{p^k \cdot n}(X) $$
as we have claimed.
\end{proof}

\section{The multiplicative property of Euler's totient function}

\label{section-phi-multiplicative}

In this section, we show how identity \eqref{our-identity} is related to the multiplicative property of $\varphi$.

We remark that comparing the degree of both sides in the Fundamental Identity \eqref{fundamental-identity} gives a well-known property of Euler's totient function
\be \sum_{d|m} \varphi(d)= n \label{property-phi}. \ee

Likewise if we compare the degree of both sides in \eqref{our-identity}, we get that 
\be \sum_{d|m} \varphi(dn) = m \varphi(n)  \; \hbox{when} \; m \perp n. \;  \label{comparison-of-degree}  \ee 

\begin{theorem}
The identity \eqref{comparison-of-degree} is equivalent to the multiplicative property of $\varphi$ 
\be \varphi(mn)= \varphi(m) \varphi(n) \; \hbox{when} \; m \perp n  \label{phi-multiplicative} \ee
in the sense that each property can be deduced from the other using \eqref{property-phi}.
\end{theorem}

\begin{proof}
Assume first that that $\varphi$ is multiplicative. Then  \eqref{comparison-of-degree} follows easily from
\eqref{property-phi} since $d|m \Rightarrow d\perp n$. Therefore,
$$  \sum_{d|m} \varphi(dn) = \sum_{d|m} \varphi(d) \varphi(n) = \varphi(n)  \sum_{d|m} \varphi(d) = m \varphi(n).$$  
On the other hand, assume that \eqref{comparison-of-degree} holds. We will show \eqref{phi-multiplicative} holds
by induction on $m$ (for every $n$ coprime with $m$). For $m=1$, it holds trivially since $\varphi(1)=1$.
Assume then \eqref{phi-multiplicative} holds for any $m^\prime<m$. Then using 
\eqref{comparison-of-degree}
$$ \sum_{d|m,d<m} \varphi(dn) + \varphi(nm) = m \varphi(n) $$
Since $d|m \Rightarrow d\perp n$ and since $d<m$, we deduce using the induction hypothesis that
$$ \sum_{d|m,d<m} \varphi(d) \varphi(n) + \varphi(nm) = m \varphi(n) $$
or
$$  \varphi(n) \sum_{d|m,d<m} \varphi(d) + \varphi(nm) = m \varphi(n) $$
But \eqref{property-phi} gives
$$ \sum_{d|m,d<m} \varphi(d) = m- \varphi(m) $$
Therefore
$$  \varphi(n) [m-\varphi(m)] + \varphi(nm) =  m \varphi(n) \Rightarrow \varphi(n) \varphi(m) = \varphi(nm) $$
By the principle of (complete) mathematical induction it follows that \eqref{phi-multiplicative} holds for every $m,n\in \N$.
\end{proof}

\section{Ramanujan sums}

In this section, we will apply \eqref{our-identity} to the Ramanujan sums \eqref{Ramanujan-sums}, and 
deduce a formula for computing the coefficients of the cyclotomic polynomials. 

We will make use of the logarithmic derivative operator
$$ L[P] = \frac{P^\prime}{P} $$
on polynomials. We observe that it has the fundamental property 
\be L[P\cdot Q] = L[P] + L[Q] \label{property-L} \ee

We will also use the method of generating functions. We need the following lemma (taken from \cite[appendix III
to chapter X]{RP}):
\begin{lemma}
Let $P\in \mathbb{C}[X]$ be a polynomial of degree $N$ with complex coefficients, 
$$ P(z)= \sum_{j=0}^N a_j z^j \quad \hbox{with} \; a_n \neq 0.  $$
Let $\rho_1,\rho_2, \ldots, \rho_N$ be the roots of $P$ (repeated according to their multiplicity) 
and let
$$ S_q = S_q[P] := \rho_1^q + \rho_2^q + \ldots + \rho_{N}^q $$
be the sum of its $q$-powers. Then $L[P]$ has the following Laurent expansion
\be L[P](z) =\frac{P^\prime(z)}{P(z)} =  \sum_{q=0}^\infty \frac{S_q}{z^{q+1}} \label{Laurent-expansion} \ee
for $|z|> M=\max_{1\leq j \leq N} |\rho_j|$.
\label{lemma-RP1}
\end{lemma}

\begin{proof}
We have that
$$ P = a_n \; (z-\rho_1)(z-\rho_2) \ldots (z-\rho_N) $$
Using \eqref{property-L}, we have that
$$ L[P](z) = \sum_{j=1}^N \frac{1}{z-\rho_j} $$
The lemma follows by expanding each term in a geometric series
$$ \frac{1}{z-\rho_j} = \frac{1}{z} \cdot \frac{1}{1-\left(\rho_j/z \right) } = 
\frac{1}{z} \sum_{q=0}^\infty \left( \frac{\rho_j}{z}\right)^q =
\sum_{q=0}^\infty \frac{\rho_j^q}{z^{q+1}} \; \hbox{for} \; |z|>|\rho_j| $$
and adding the results (which is legitimate for $|z|>M$ by the absolute convergence of the series).
\end{proof}

By applying this lemma to the cyclotomic polynomial $\Phi_n(z)$ we immediately get 

\begin{corollary}
We have the following Laurent expansion for the logarithmic derivative of the cyclotomic polynomials:
$$ L[\Phi_n](z) = \frac{\Phi^\prime_n(z)}{\Phi_n(z)} =  \sum_{q=0}^\infty \frac{c_n(q)}{z^{q+1}} \; \hbox{for} \; |z|>1. $$
\end{corollary}

\begin{remark}
Let $P$ be a polynomial and let $Q(z)=P(z^m)$. Then 
$$ L[Q](z) = m\; z^{m-1} \; L[P](z^m). $$
\end{remark}

We are ready to see how property \eqref{our-identity} applies to the Ramanujan sums:

\begin{proposition}
\eqref{our-identity} implies that if $n \perp m$, 
$$\sum_{d|m} c_{dn}(q) 
= \left\{ \begin{array}{ll}
 m \cdot c_n(q/m)  & \hbox{if} \; m|q \\
  0  & \hbox{otherwise} 
\end{array} \right. $$
\end{proposition}

\begin{proof}
We consider the identity in Theorem \ref{theorem-our-identiy}. By taking the logarithmic derivative on both sides 
and using the previous remark, we get for $|z|>1$,
$$ m \; z^{m-1} L[\Phi_n](z^m) = \sum_{d|m} L[\Phi_{d\cdot n}](z). $$
We expand each side in a Laurent series
$$ \sum_{r=0}^\infty \frac{m \; c_n(r)}{z^{(r+1)m-(m-1)}} = \sum_{d|m} \sum_{q=0}^\infty \frac{c_{dn}(q)}{z^{q+1}}
=  \sum_{q=0}^\infty \left( \sum_{d|m} c_{dn}(q)  \right) \frac{1}{z^{q+1}}. $$
By the uniqueness of the Laurent expansion,
$$  \sum_{d|m} c_{dn}(q) = m \cdot c_n(r)  $$
when $ (r+1)m-(m-1) = q+1 \Leftrightarrow rm = q$, and that the sum is zero otherwise.
\end{proof}

\begin{remark}
When $q=0$ this property reduces to \eqref{comparison-of-degree}, since $c_n(0)=\varphi(n)$. 
\end{remark}

As before, using the additive version of Möbuis inversion formula, we get
\begin{corollary}
If $n \perp m$, 
$$ c_{mn}(q) = \sum_{d|\gcd(m,q)} d\; c_n\left(\frac{q}{d}\right) \; \mu\left(  \frac{m}{d} \right).  $$
\end{corollary}

In particular, if we choose $n=1$, $c_1(q/d)=1$ and we get the following known explicit formula for the Ramanujan sums due to Kluyver \cite{kluyver1906} See also \cite[Theorem 8.6]{Apostol}
\be c_{m}(q) = \sum_{d|\gcd(m,q)} d\; \mu\left(\frac{m}{d}\right) \quad \forall m \in \N  \label{explicit-Ramanujan}. 
\ee
Another explicit formula for the Ramanujan sums is 
\be c_m(q) = \frac{\mu\left( \frac{m}{\gcd(m,q)}\right) \varphi(m)}{\varphi\left( \frac{m}{\gcd(m,q)}\right)} 
\label{explicit-Ramanujan2} \ee 
The function on the right hand side was initially studied by von Sterneck \cite{von1902sitzungsber}. Later, Kluyver \cite{kluyver1906} and also Hölder \cite{holder1936theorie} proved that it coincides with the Ramanujan sums. See \cite{Ramanujan-supercharacters} for more information on the Ramanujan sums and their history.

\medskip

We conclude this note by explaining how the coefficients of the cyclotomic polynomial $\Phi_n(X)$ can be recursively computed using the Ramanujan sums, without the need of factoring polynomials.

\begin{lemma}[Newton Relations]
Let $P\in \mathbb{C}[X]$ of degree $N$ and consider the sums $S_q$ of the $q$-powers of its roots as in lemma \ref{lemma-RP1}. Then the coefficients $a_j$ of $P$ are related to the sums $S_q$ by:
$$ a_{N-\ell} =  -\frac{1}{\ell} \; \sum_{j=0}^{\ell-1} a_{N-j} \cdot S_{\ell-j}  \; \hbox{for} \; j=1,1,2,\ldots, N-1 $$
\end{lemma}

This result follows from Lemma \ref{lemma-RP1} by writting \eqref{Laurent-expansion} as
$$  \sum_{r=1}^N r \; a_r  \; z^{r-1} = \left( \sum_{j=0}^N a_j z^{j} \right) \cdot
\left( \sum_{q=0}^\infty \frac{S_q}{z^{q+1}} \right) $$
and equating the coefficients on both sides. See \cite[appendix III to chapter X]{RP} for details\footnote{Beware that 
in this book the notation for the coefficient of $X^j$ in $P$ is $a_{N-j}$ instead of $a_j$.}.

\begin{corollary} 
Let $\Phi_n(X)$ be the cyclotomic polynomial. Its coefficients $a_j(n)$ (for $0\leq j \leq N=\varphi(n)$)  
can be recursively computed in terms of the Ramanujan sums using the relation
$$ a_{N-\ell}(n) =  -\frac{1}{\ell} \; \sum_{j=0}^{\ell-1} a_{N-j}(n) \cdot c_n(\ell-j)  \quad \hbox{for} \; \ell= 1,2,\ldots,N-1, $$
starting from
$$ a_{N}=1 $$
Together with \eqref{explicit-Ramanujan} or \eqref{explicit-Ramanujan2} these formulas provide an algorithm for computing $\Phi_n(X)$ without the need of dividing polynomials. 
\end{corollary}

More information on the coefficients of cyclotomic polynomials and their relations to other arithmetical functions can be found in \cite{herrera2020coefficients}. 
We also refer those readers who are interested in efficient algorithms for the computation of cyclotomic 
polymonials to \cite{Arnold-Monagan} and \cite{Brent}.

\medskip

\printbibliography

\end{document}